\documentclass[11pt]{amsart}
\usepackage{amssymb}
\usepackage{amsmath}
\usepackage{mathrsfs}
\usepackage{amsfonts}
\usepackage{txfonts}
\usepackage{amssymb}
\usepackage{color}

\newcommand\RR{{{\mathbb R}}}

\newcommand\cL{{\mathcal L}}

\newcommand\re{{\mathcal Re}}

\newtheorem{theo}{Theorem}[section]
\newtheorem{lemm}[theo]{Lemma}

\newtheorem{rema}[theo]{Remark}

\begin{document}

\title[Hypoellipticity effects for Landau]
{Landau linearized operator and hypoellipticity}

\author{R. Alexandre}
\address{R. Alexandre,
Department of Mathematics, Shanghai Jiao Tong University
\newline\indent
Shanghai, 200240, P. R. China
\newline\indent
and \newline\indent Irenav, Arts et Metiers Paris Tech, Ecole
Navale,
\newline\indent
Lanveoc Poulmic, Brest 29290 France }
\email{radjesvarane.alexandre@ecole-navale.fr}


\date{February 2011}


\begin{abstract}
We consider the linearized Landau operator for which we provide simple proofs of hypoellipticity, and in particular we recover the recent results of H\' erau and Pravda-Starov \cite{herau-all}. Our arguments are elementary and in particular avoids any use of pseudo-differential calculus.
\end{abstract}

\maketitle


\section{Introduction}

We consider hypoelliptic effects associated with the time version of a kinetic equation related to the linearized Landau equation and studied recently by Herau and Pravda-Starov \cite{herau-all}
\begin{equation}\label{eq-1}
\partial_t f +v.\nabla_x f - \nabla_v .\lambda (v) \nabla_v f - (v\wedge \nabla_v ).\mu (v) (v\wedge \nabla_v f) +F(v) f(v) = h
\end{equation}
where $t\in \RR$, $x\in \RR^3$ and $v\in \RR^3$. $f$ and $h$ will be supposed to be in $L^2$, where here and below $L^2$ denotes the usual space w.r.t. full variables $(t,x,v)$. In fact, as it will be clear from the proofs, it is also possible to work in weighted $L^2$ spaces and even in weighted Sobolev spaces. The norm in $L^2$ will be denoted by $\| .\|$ and its associated scalar product by $(.\  , .)$. We shall work with real functions $f$ and $h$ though there is absolutely no difficulties in considering complex cases, up to the addition of the real parts when necessary. 

As regards the coefficients appearing in \eqref{eq-1}, we assume that the positive functions $\lambda$, $\mu$ and $F$ satisfy the following coercive type lower bounds
\begin{equation}\label{hyp-1}
 \lambda (v) \gtrsim <v>^\gamma , \ \mu (v) \gtrsim <v>^\gamma \mbox{ and } F(v) \gtrsim <v>^{\gamma +2}
 \end{equation}
where $-3 \leq \gamma \leq 1$, and
\begin{equation}\label{hyp-2}
| D^m_v \lambda | \lesssim <v>^{\gamma -m} , | D^m_v \mu | \lesssim <v>^{\gamma -m} \mbox{ and }| D^m_v F(v)| \lesssim <v>^{\gamma +1 -m}
\end{equation}
for any $m$ of order at most  two.

We set
\begin{equation}\label{def-op}
\cL (f) = - \nabla_v .\lambda (v) \nabla_v f - (v\wedge \nabla_v ).\mu (v) (v\wedge \nabla_v f) +F(v) f(v) 
= \cL_1 (f) +\cL_2 (f) +\cL_3(f)
\end{equation}
so that \eqref{eq-1} writes also
\begin{equation}\label{eq-2}
\partial_t f +v.\nabla_x f  +\cL (f)  = h.
\end{equation}
As explained in \cite{herau-all}, \eqref{eq-1} or its version \eqref{eq-2} is related to the linearized Landau operator, which plays a crucial role in Plasma Physics, see for example \cite{alex-review,villani} and the references therein.  Moreover the above assumptions \eqref{hyp-1} and \eqref{hyp-2} are natural in view of Guo's work \cite{guo}. However, note that the model \eqref{eq-1} does not take into account the kernel which is naturally associated with the true Landau linearized equation. In particular, for applications to the true nonlinear Landau model near equilibrium, the present results need to be adapted, though the main issue is related to the macroscopic part.

Our main concern is to give a shorter proof of the following result about optimal hypoelliptic results which was first proven in \cite{herau-all} (in the time independent version)

\begin{theo} [Herau and Pravda-Starov \cite{herau-all}] Assume $f$ and $h$ belong to $L^2$. Then, under the hypothesis \eqref{hyp-1} and \eqref{hyp-2}, we have
$$
\left\{\begin{array}{c} \| <v>^{\gamma +2} f\|^2 + \| <v>^\gamma |D_v|^2 f\|^2 + \| <v>^\gamma |v\wedge \nabla_v |^2 f\|^2 \\ +  \| <v>^{\gamma\over 3} |D_x |^{2\over 3} f\|^2 + \| <v>^{\gamma\over 3} |v\wedge \nabla_x |^{2\over 3}f \|^2 \lesssim \| f\|^2 +\| h\|^2 . \end{array}\right.$$
\end{theo}

The proof in \cite{herau-all} uses pseudo-differential calculus, together with Wick calculus. In the continuation of a previous work of Morimoto and Xu \cite{morimoto-all}, a similar study was also performed in \cite{lerner-all} in order to deduce hypoelliptic results for a fractional order kinetic equation, and again the proof therein was using such tools.

Recently, we have provided in \cite{alex-hypo} a very simple proof of the results of \cite{lerner-all} by using arguments originally introduced by Bouchut \cite{bouchut} and Perthame \cite{perthame}.

Again herein, we shall give a different proof of this main Theorem by using simple and standard arguments. Moreover, we keep the regularity of the coefficients as low as possible, this point being connected with the second order commutators estimations which are needed in the proofs. As a byproduct, other estimations will emerge from our computations. Note also that, as usual, time derivatives estimates are also available, but we do not detail this point. All in all, together with our previous work \cite{alex-hypo}, we provide extremely simple arguments to deduce hypoelliptic results for kinetic equations with a diffusive part. It is to be expected that the underlying arguments are sufficiently simple to enable the study of different questions related to diffusive models arising from scaled kinetic equations. We hope to get back on this issue soon. Furthermore as it will be clear from the proofs, the Cauchy problem can be also analyzed with the same methods up to some minor changes. It is expected that such simple methods will provide other methods for the analysis of the Cauchy problem associated with fully nonlinear kinetic equations such as Boltzmann or Landau equations, see the quoted works in the bibliography.

We shall always assume that all functions $f$ and $h$ are smooth.The paper is organized as follows.  We first deduce in Section 2 some estimates from a transport equation. These are used in Section 3, in particular to control cross products terms.

\section{Preliminary results}

First of all, by multiplying the equation by $f$, integration over all variables and using the assumptions \eqref{hyp-1}, we get

\begin{lemm}\label{lemme-1} One has
$$\| <v>^{\gamma \over 2} \nabla_v f\|^2 + \| <v>^{\gamma\over 2} (v\wedge \nabla_v f) \|^2 + \| <v>^{{\gamma\over 2} +1} f\|^2 \lesssim \| f\| \ \| h\| .$$
\end{lemm}

The next two Lemmas are the adaptions of some of the steps which appear in Bouchut's paper \cite{bouchut}. They are related to transport type equations, with a given right hand side, and assuming that we know already some kind of regularity w.r.t. velocity variable, they give some informations about regularity of some spatial derivatives.

\begin{lemm}\label{lemme-2} Assume $f\in L^2$, $g\in L^2$, $<v>^{\gamma} | D_v|^2 f\in L^2$ and
$$\partial_t f + v.\nabla_x f =g .$$
Then
$$ \| <v>^{\gamma\over 3} |D_x|^{2\over 3} f\| \lesssim \bigg\{  \| <v>^{\gamma -1} f\|^{1\over 3} + \|<v>^\gamma |D_v|^2 f\|^{1\over 3} \bigg\}\  \| g\|^{2\over 3}.$$
\end{lemm}

Observe that the first term in the r.h.s. of this estimate is controlled by the $L^2$ norm of $f$, since $\gamma \leq 1$. 
\begin{proof}
Following \cite{bouchut}, we start from the formula
$$\partial_{x_j} f = \partial_{v_j} g - (\partial_t +v.\nabla_x )\partial_{v_j} f.$$
Then we write
$$\| <v>^{\gamma\over 3} |D_x|^{2\over 3} f\|^2 = \| <v>^{\gamma \over 3} | D_x|^{- 1\over 3} \partial_{x_j} f \|^2$$
$$ = ( <v>^{{2\gamma}\over 3} | D_x |^{-2\over 3} \partial_{x_j} \bar f , \partial_{x_j} f  )  = (<v>^{2\gamma\over 3} |D_x|^{-2\over 3} \partial_{x_j} \bar f , \partial_{v_j} g - (\partial_t +v.\nabla_x )\partial_{v_j} f)$$
$$ = - 4{\gamma \over 3 }( <v>^{2\gamma \over 3 -1} v.\nabla_x (|D_x|^{-2\over 3} f , g) - (<v>^{2\gamma\over 3} |D_x|^{-2\over 3} \partial_{x_j}\partial_{v_j} \bar f ,g) + ( <v>^{2\gamma\over 3} |D_x|^{-2\over 3} \partial_{x_j} \bar g , \partial_{v_j} f ) .$$

Thus
\begin{equation}\label{ineq-1}
\| <v>^{\gamma\over 3} |D_x|^{2\over 3} f\|^2  = - 4{\gamma \over 3} ( <v>^{2{\gamma \over 3} -1} v.\nabla_x (|D_x|^{-2\over 3} f , g)  - 2 \re (<v>^{2\gamma\over 3} |D_x|^{-2\over 3} \partial_{x_j}\partial_{v_j} \bar f ,g) =I+II.\end{equation}

Starting from
$$II^2 \lesssim \| <v>^{2\gamma\over 3} |D_x|^{-2\over 3} \partial_{x_j}\partial_{v_j}  f \|^2 \ \| g\|^2 ,$$

it follows that
$$\| <v>^{2\gamma\over 3} |D_x|^{-2\over 3} \partial_{x_j}\partial_{v_j}  f \|^2  \lesssim (<v>^{4\gamma \over 3} \|D_x|^{2\over 3} \partial_{v_j} f , \partial_{v_j} f )$$
$$ = + (<v>^{4\gamma \over 3} \|D_x|^{2\over 3}  f ,|D_v|^2 f )  +  8\gamma /3 ( <v>^{{{4\gamma}\over 3} - 1} v_j |D_x|^{2\over 3} f , \partial_{v_j} f) = A+B.$$

Now, one has
$$A \lesssim \| <v>^{\gamma\over 3} |D_x|^{2\over 3} f\| \ \|  <v>^\gamma | D_v|^2 f\|$$
while for $B$, we have
$$ B = 8{\gamma \over 3} ( <v>^{{{4\gamma}\over 3} - 1}  |D_x|^{2\over 3} f , v.\nabla_v f)  = 8{\gamma \over 3} ( <v>^{{{4\gamma}\over 3} - 1} [ |D_x|^{1\over 3} f ], v.\nabla_v  [ |D_x|^{1\over 3} f ]) $$
which is of the form
$$ ( g, \beta (v) .\nabla_v g) \simeq (g,g \ div_v [\beta ])$$
where $\beta (v) = <v>^{{{4\gamma}\over 3} - 1} v$. Therefore $div \ \beta \simeq <v>^{{{4\gamma}\over 3} - 1}  + <v>^{{{4\gamma}\over 3} - 3}  <v>^2$. Thus 
$$B \lesssim (  <v>^{{{4\gamma}\over 3} - 1}  g, g) = ( <v>^{{{4\gamma}\over 3} - 1}  [ |D_x|^{1\over 3} f ] , [ |D_x|^{1\over 3} f ])$$
$$ = (<v>^{\gamma\over 3} |D_x|^{2\over 3} f, <v>^{\gamma -1} f) \lesssim \| <v>^{\gamma\over 3} |D_x|^{2\over 3} f \|  \ \| <v>^{\gamma -1} f \|$$
and
$$A+B \lesssim \| <v>^{\gamma\over 3} |D_x|^{2\over 3} f \| \ \bigg\{ \|<v>^\gamma |D_v|^2 f\|  + \| <v>^{\gamma -1} f \| \bigg\} .$$
Therefore
$$II \lesssim  \| <v>^{\gamma\over 3} |D_x|^{2\over 3} f\|^{1\over 2}  \bigg\{ \| <v>^{\gamma -1} f\|^{1\over 2} + \|<v>^\gamma |D_v|^2 f\|^{1\over 2} \bigg\}\ \| g\| .$$

For $I$, (using Fourier transform w.r.t. variable $x$)
$$I \lesssim \|<v>^{2{\gamma\over 3} - {1\over 2}} |D_x|^{1\over 3} f \| \ \| g\|$$
and by the same computations, we get
$$I \lesssim \| <v>^{\gamma\over 3} |D_x |^{2\over 3} f \|^{1\over 2}\  \| <v>^{\gamma -1} f\|^{1\over 2}\  \|g\| .$$
Thus
$$I+II \lesssim  \| <v>^{\gamma\over 3} |D_x|^{2\over 3} f\|^{1\over 2}\  \bigg\{  \| <v>^{\gamma -1} f\|^{1\over 2} + \|<v>^\gamma |D_v|^2 f\|^{1\over 2} \bigg\}\  \| g\|$$
and finally in view of \eqref{ineq-1}
$$ \| <v>^{\gamma\over 3} |D_x|^{2\over 3} f\| \lesssim \bigg\{ \| <v>^{\gamma -1} f\|^{1\over 3} + \|<v>^\gamma |D_v|^2 f\|^{1\over 3} \bigg\} \  \| g\|^{2\over 3}.$$

This ends the proof of the Lemma.
\end{proof}

\begin{lemm}\label{lemme-3} Assume $f\in L^2$, $<v>^{\gamma +2}f\in L^2$, $g\in L^2$, $<v>^\gamma | v\wedge D_v|^2 f\in L^2$ and
$$\partial_tf +v.\nabla_x f =g .$$
Then
$$\| <v>^{\gamma\over 3} <v\wedge D_x>^{2\over 3} f\| \lesssim  \| <v>^{\gamma\over 3} f\| + $$
$$ + \| g\|^{2\over 3} \bigg\{  \| <v>^{\gamma +2} f\|^{1\over 3} + \|<v>^\gamma |D_v|^2 f\|^{1\over 3} + \|<v>^\gamma |v\wedge D_v|^2 f\|^{1\over 3} \bigg\}  .$$
\end{lemm}

%

\begin{proof} We want to estimate 
$$\| <v>^{\gamma\over 3} <v\wedge D_x > ^{2\over 3} f\|^2 = \| <v>^{\gamma\over 3} <v\wedge k >^{2\over 3} \hat f \|^2$$
and, in view of Lemma \ref{lemme-2}, it is enough to work for $<v\wedge k>\geq < k>$, ie $<v\wedge k> <k>^{-1} \geq 1$ (using Parseval relation w.r.t. variable $x$). Let $\phi$ be a positive function which is $0$ for small values and $1$ for large values. To simplify notations, write $\tilde\phi (v,k) = \phi (<v\wedge k> <k>^{-1})$. Then we need to estimate
 $$\| <v>^{\gamma\over 3} \tilde\phi (v,D_x) <v\wedge D_x > ^{2\over 3} f\|^2 .$$

We write
$$\| <v>^{\gamma\over 3} \tilde\phi (v,D_x) <v\wedge D_x > ^{2\over 3} f\|^2  = ( <v>^{\gamma\over 3}  \tilde\phi (v,D_x) <v\wedge D_x > ^{2\over 3} \bar f, <v>^{\gamma\over 3}  \tilde\phi (v,D_x) <v\wedge D_x > ^{2\over 3}  f)$$
$$(<v>^{2\gamma\over 3} \tilde\phi^2 (v,D_x) <v\wedge D_x>^{2\over 3} \bar f, f) = (<v>^{2\gamma\over 3} \tilde\phi^2 (v,D_x) <v\wedge D_x>^{-2\over 3} <v\wedge D_x>^2 \bar f ,f)$$
$$ = (<v>^{2\gamma\over 3} \tilde\phi^2 (v,D_x) <v\wedge D_x>^{-2\over 3}  \bar f ,f) + (<v>^{2\gamma\over 3} \tilde\phi^2 (v,D_x) <v\wedge D_x>^{-2\over 3} |v\wedge D_x|^2 \bar f ,f) .$$

Note that the first term is (for example) bounded by $\| <v>^{\gamma\over 3} f\|^2$. So we concentrate on the second one
$$Imp = (<v>^{2\gamma\over 3} \tilde\phi^2 (v,D_x) <v\wedge D_x>^{-2\over 3} |v\wedge D_x|^2 \bar f ,f)$$
$$ = - ( (<v>^{2\gamma\over 3} \tilde\phi^2 (v,D_x) <v\wedge D_x>^{-2\over 3} v\wedge D_x \bar f ,v\wedge D_x f).$$

Now we use the fact that $\nabla_x f = \nabla_v g - v.\nabla_x \nabla_v f$, and therefore $v\wedge \nabla_x f = (v\wedge \nabla_v g) - v.\nabla_x (v\wedge \nabla_v f)$. Thus
$$Imp = - (<v>^{2\gamma\over 3}\tilde\phi^2 (v,D_x) |v\wedge D_x|^{-2\over 3} v\wedge \nabla_v\bar g , v\wedge \nabla_x f) + (<v>^{2\gamma\over 3}\tilde\phi^2 (v,D_x) |v\wedge D_x|^{-2\over 3} v.\nabla_x (v\wedge \nabla_v \bar f ), v\wedge \nabla_x f )$$
$$= - (<v>^{2\gamma\over 3}\tilde\phi^2 (v,D_x) |v\wedge D_x|^{-2\over 3} v\wedge \nabla_v\bar g , v\wedge \nabla_x f) - (<v>^{2\gamma\over 3}\tilde\phi^2 (v,D_x) |v\wedge D_x|^{-2\over 3} v\wedge \nabla_v \bar f , v\wedge \nabla_x g ) .$$

Set $S = <v>^{2\gamma\over 3}\tilde\phi^2 (v,D_x) |v\wedge D_x|^{-2\over 3}$.Then
$$ Imp =  - (S v\wedge \nabla_v\bar g , v\wedge \nabla_x f) - (S v\wedge \nabla_v \bar f , v\wedge \nabla_x g ) .$$

We introduce some notations (though there is also another line of proof which avoids such notations. However, the arguments are simple enough). Let $e_j$ be the canonical basis of $\RR^3$. Then we can write
$$v\wedge \nabla_x F = (v\wedge \nabla_x F).e_j e_j = - [(v\wedge e_j).\nabla_x F ]e_j .$$

Setting  $X_j =X_j (v,\partial_x) F = (v\wedge e_j).\nabla_x F$, we get $v\wedge \nabla_x F = - X_j (F) e_j$. Similarly, we can write
$$v\wedge\nabla_v G = (v\wedge \nabla_v G).e_j e_j = -(v\wedge e_j).\nabla_vG e_j .$$

Letting $V_j = V_j (v,\partial_v) G = (v\wedge e_j).\nabla_vG$, then $v\wedge\nabla_v G = - V_j (G)e_j$. Note that $X^\ast_j = -X_j$ and $V^\ast_j = -V_j$. Now we can write (with summation of indices)
\begin{equation}\label{ineq-2} \left\{\begin{array}{c} 
 Imp =  - (S V_j (\bar g) . X_j (f)) - (S V_j (\bar f ) . X_j (g) ) =  - (S V_j (\bar g) . X_j (f)) + (X_jS V_j (\bar f ) . g )\\ 
=   (\bar g . V_jSX_j (f)) + (SX_j V_j (\bar f ) . g ) = 2\re (\bar g , SX_jV_j (f)) + (\bar g , [V_j , SX_j] (f)) \\
 = Imp_1 +Imp_2 . \end{array}\right.
 \end{equation}

Let us look to $Imp_2$. We use Fourier transform w.r.t. $x$ variables. Then (summation over j)
$$\widehat{[V_j , SX_j] (f)} = i V_ j [ S (k,v) (v\wedge e_j).k \hat f(k)] - iS(k,v) (v\wedge e_j). k (V_j (\hat f))$$
$$ = i V_ j [ S (k,v) (v\wedge e_j).k ]\hat f(k)$$
where
$$S(k,v)= <v>^{2\gamma\over 3} \phi^2 (<v\wedge k> <k>^{-1} ) |v\wedge k|^{-2\over 3} .$$

Set below $\psi =\phi^2$ also. Note that $(v\wedge e_j).k = (e_j\wedge k).v$. Then $\nabla_v [(v\wedge e_j).k ] = \nabla_v [(e_j\wedge k).v ] = e_j\wedge k$ and $V_j [v\wedge e_j).k] = (v\wedge e_j).(e_j\wedge k).$. Note that $V_j ((e_j\wedge k).v ) = (v\wedge e_j ). (e_j \wedge k) = -(v\wedge (e_j \wedge k)) .e_j$ and $v\wedge (e_j\wedge k) = (v.k).e_j - (v.e_j)k$. Thus $-V_j ((e_j\wedge k).v )  = (v.k) - (v.e_j)(k.e_j) =0$. 

Next $\nabla_v <v>^{2\gamma\over 3} \sim <v>^{{2\gamma\over 3} -2} v$, and therefore $V_j \nabla_v <v>^{2\gamma\over 3} =0$. On the whole
$$V_ j [ S (k,v) (v\wedge e_j).k ]\ = <v>^{2\gamma\over 3} (v\wedge e_j).k V_j [ \psi (<v\wedge k> <k>^{-1} )) |v\wedge k|^{-2\over 3}] .$$

Now note that $ |v\wedge k|^{-2\over 3} = (|v|^2|k|^2 - (v.k)^2 )^{-1\over 3}$, and thus
$$\nabla_v [|v\wedge k|^{-2\over 3} ] = - {1\over 3} (|v|^2|k|^2 - (v.k)^2 )^{-4\over 3} [ |k|^2 2v - 2(v.k) k] \mbox{ and } V_j  [|v\wedge k|^{-2\over 3} ]  = {2\over 3}  |v\wedge k|^{-8\over 3} (v.k) (v\wedge e_j).k .$$

Finally
$$\nabla_v [ \psi (<v\wedge k> <k>^{-1} )) ] = \psi '( <v\wedge k> <k>^{-1} )) <k>^{-1} \nabla_v [ 1+ |v\wedge k|^2]^{1\over 2}$$
$$ = \psi '( <v\wedge k> <k>^{-1} )) <k>^{-1} [ 1+ |v\wedge k|^2]^{-1\over 2} [ |k|^2 2v - 2(v.k) k] .$$
Thus
$$V_j  [ \psi (<v\wedge k> <k>^{-1} )) ]  =-2  \psi '( <v\wedge k> <k>^{-1} )) <k>^{-1} <v\wedge k>^{-1} (v.k) (v\wedge e_j).k$$
and at last
\begin{equation}\label{ineq-3} \left\{ \begin{array}{c}
V_ j [ S (k,v) (v\wedge e_j).k ] 
 = <v>^{2\gamma\over 3} \psi (<v\wedge k> <k>^{-1} )) {2\over 3}  |v\wedge k|^{-8\over 3} (v.k) |v\wedge k|^2 \\ 
 -2 <v>^{2\gamma\over 3} \psi '( <v\wedge k> <k>^{-1} )) <k>^{-1} <v\wedge k>^{-1} (v.k)  |v\wedge k|^2  |v\wedge k|^{-2\over 3} . \end{array}\right.
 \end{equation}

By definition of $\psi '$, the second term in \eqref{ineq-3} is bounded by (because $|v\wedge k|\sim |k|$) $ <v>^{2{\gamma\over 3} +1} <k>^{1\over 3}$, and therefore going back we have a contribution to $Imp_2$ as $\| <v>^{2{\gamma\over 3} +1} <k>^{1\over 3} \hat f \| \ \| g\|$. Since
$$\| <v>^{2{\gamma\over 3} +1} <k>^{1\over 3} \hat f \|^2 = (<v>^{4{\gamma\over 3} +2} <k>^{2\over 3} \hat f , \hat f) $$
$$= (<v>^{\gamma\over 3} <k>^{2\over 3}  \hat f , <v>^{\gamma +2} \hat f) \lesssim \| <v>^{\gamma\over 3} |D_x|^{2\over 3} f \| \ \| <v>^{\gamma +2} f\|$$

the second term in \eqref{ineq-3} gives a contribution to $Imp_2$ as
$$ \| <v>^{\gamma\over 3} |D_x|^{2\over 3} f \|^{1\over 2}\  \| <v>^{\gamma +2} f\|^{1\over 2}  \ \|g\| .$$

For the first term in \eqref{ineq-3}, it is also bounded by the same form. All in all, we have shown that
\begin{equation}\label{ineq-4}
Imp_2 \lesssim  \| <v>^{\gamma\over 3} |D_x|^{2\over 3} f \|^{1\over 2} \ \| <v>^{\gamma +2} f\|^{1\over 2}  \ \|g\| .
\end{equation}

Now we can turn to $Imp_1$ from \eqref{ineq-2} to get first of all
$$Imp_1 \lesssim \| g\| \  \| SX_j V_j (f) \| .$$
Then, it follows that
\begin{equation}\label{ineq-5} \left\{ \begin{array}{c}
 \| SX_j V_j (f) \|^2 = ( SX_j V_j (f), SX_j V_j (f) )= -(S^2 X_j^2 V_j (f), V_j(f)) \\ 
= -([S^2 X_j^2, V_j] (f), V_j(f)) + (S^2X_j^2 (f) , V_j^2 (f) =A_1+B_1 . \end{array}\right.
\end{equation}

$B_1$ is estimated as follows
$$B_1 \lesssim \| <v>^{-\gamma}S^2X_j^2 (f)  \| \ \| <v>^\gamma V_j^2 (f)\| \lesssim \| <v>^{\gamma\over 3} |v\wedge D_x|^{2\over 3} f\| \  \| <v>^\gamma |v\wedge D_v|^2 f\| .$$

For $A_1$, this is again a commutator estimation: by Fourier transform w.r.t. $x$, we have, using the previous computations
$$ \widehat{ [ S^2X_j^2 ,V_j] (f) } = -iS(k,v) (v\wedge e_j).k) V_j [ S(k,v) (v\wedge e_j).k ] \hat f (k)$$
$$= -i <v>^{2\gamma\over 3} \phi^2 (<v\wedge k><k>^{-1} ) |v\wedge k|^{-2\over 3} (v\wedge k).e_j V_j [ S(k,v) (v\wedge e_j).k ] \hat f (k)$$
with
$$V_ j [ S (k,v) (v\wedge e_j).k ]\ = <v>^{2\gamma\over 3} (v\wedge e_j).k V_j [ \psi (<v\wedge k> <k>^{-1} )) |v\wedge k|^{-2\over 3}]$$
$$ = <v>^{2\gamma\over 3} \psi (<v\wedge k> <k>^{-1} )) {2\over 3}  |v\wedge k|^{-8\over 3} (v.k) |v\wedge k|^2 $$
$$ -2 <v>^{2\gamma\over 3} \psi '( <v\wedge k> <k>^{-1} )) <k>^{-1} <v\wedge k>^{-1} (v.k)  |v\wedge k|^2  |v\wedge k|^{-2\over 3} .$$

Thus we have two contributions, the first one being given by
$$A_{11} = -i <v>^{2\gamma\over 3} \phi^2 (<v\wedge k><k>^{-1} ) |v\wedge k|^{-2\over 3} (v\wedge k).e_j <v>^{2\gamma\over 3} \psi (<v\wedge k> <k>^{-1} )) {2\over 3}  |v\wedge k|^{-8\over 3} (v.k) |v\wedge k|^2$$
$$\simeq i <v>^{4\gamma\over 3} \phi^4 (<v\wedge k><k>^{-1} ) |v\wedge k|^{-4\over 3} (v\wedge k).e_j (v.k)$$
which is bounded from above by $<v>^{4{\gamma\over 3}+1} |k|^{2\over 3}$. Thus we may write
$$A_{11} \simeq i <v>^{{\gamma\over 3} -1} \phi^4 (<v\wedge k><k>^{-1} ) |v\wedge k|^{-4\over 3} (v\wedge k).e_j (v.k) \times <v>^{\gamma +1}$$
and the contribution given by $A_{11}$ is estimated by
$$\| <v>^{\gamma\over 3} |D_x|^{2\over 3} f\| \ \| <v>^{\gamma +1} (v\wedge D_v )f\| .$$
Then
$$ \| <v>^{\gamma +1} (v\wedge D_v )f\|^2 = (<v>^{\gamma} |v\wedge D_v|^2 f , <v>^{\gamma +2}f ) \lesssim \| <v>^\gamma |v\wedge D_v|^2 f \| \ \| <v>^{\gamma +2} f\|$$
and therefore the contribution by $A_{11}$ is estimated by
$$\| <v>^{\gamma\over 3} |D_x|^{2\over 3} f\| \  \| <v>^\gamma |v\wedge D_v|^2 f \|^{1\over 2} \  \| <v>^{\gamma +2} f\|^{1\over 2} .$$

Now we turn to the other contribution in $A_1$. We have
$$A_{12} \simeq 2i <v>^{2\gamma\over 3} \phi^2 (<v\wedge k><k>^{-1} ) |v\wedge k|^{-2\over 3} (v\wedge k).e_j  <v>^{2\gamma\over 3} \psi '( <v\wedge k> <k>^{-1} ))$$
$$\times <k>^{-1} <v\wedge k>^{-1} (v.k)  |v\wedge k|^2  |v\wedge k|^{-2\over 3}$$
$$\simeq 2i <v>^{4\gamma\over 3} \phi^2 (<v\wedge k><k>^{-1} )\psi '( <v\wedge k> <k>^{-1} )) <v\wedge k>^{-1}<k>^{-1} (v.k) (v\wedge k).e_j |v\wedge k|^{2\over 3} $$
$$\simeq 2i <v>^{{\gamma\over 3}-1} \phi^2 (<v\wedge k><k>^{-1} )\psi '( <v\wedge k> <k>^{-1} )) <v\wedge k>^{-1}<k>^{-1} (v.k) (v\wedge k).e_j |v\wedge k|^{2\over 3}. <v>^{\gamma +1}$$
Then, the contribution by this term is estimated by (or by $|D_x|$ instead of $v\wedge D_x$)
$$ \| <v>^{\gamma\over 3} |v\wedge D_x|^{2\over 3} f\|\  \| <v>^{\gamma +1} v\wedge D_v f\|$$
and therefore the total contribution by $A_{11}+A_{12}$ gives
$$A_1\lesssim \| <v>^{\gamma\over 3} |D_x|^{2\over 3} f\| \  \| <v>^\gamma |v\wedge D_v|^2 f \|^{1\over 2} \ \| <v>^{\gamma +2} f\|^{1\over 2} .$$

Thus 
$$A_1+B_1 \lesssim  \| <v>^{\gamma\over 3} |D_x|^{2\over 3} f\| \ \bigg\{ \| <v>^\gamma |v\wedge D_v|^2 f \|^{1\over 2} \  \| <v>^{\gamma +2} f\|^{1\over 2} + \| <v>^\gamma |v\wedge D_v|^2 f \|\bigg\}$$ 
which gives
$$Imp_1 \lesssim \| g\| \Bigg\{  \| <v>^{\gamma\over 3} |D_x|^{2\over 3} f\| \  \bigg\{ \| <v>^\gamma |v\wedge D_v|^2 f \|^{1\over 2} \  \| <v>^{\gamma +2} f\|^{1\over 2} + \| <v>^\gamma |v\wedge D_v|^2 f \| \bigg\} \Bigg\}^{1\over 2}$$
that is
$$Imp_1 \lesssim  \| g\|  \ \| <v>^{\gamma\over 3} |D_x|^{2\over 3} f\|^{1\over 2} \  \bigg\{ \| <v>^\gamma |v\wedge D_v|^2 f \|^{1\over 4} \  \| <v>^{\gamma +2} f\|^{1\over 4} + \| <v>^\gamma |v\wedge D_v|^2 f \|^{1\over 2} \bigg\} $$
and therefore
$$Imp \lesssim \| g\|  . \| <v>^{\gamma\over 3} |D_x|^{2\over 3} f\|^{1\over 2} \ \bigg\{ \| <v>^\gamma |v\wedge D_v|^2 f \|^{1\over 4} \  \| <v>^{\gamma +2} f\|^{1\over 4} + \| <v>^\gamma |v\wedge D_v|^2 f \|^{1\over 2} + \| <v>^{\gamma +2} f\|^{1\over 2}\bigg\} .$$

Thus
$$\| <v>^{\gamma\over 3} \tilde \phi (v,D_v) <v\wedge D_x>^{2\over 3} f\| \lesssim \| <v>^{\gamma\over 3} f\| + $$
$$+ \| g\|^{1\over 2}  \ \| <v>^{\gamma\over 3} |D_x|^{2\over 3} f\|^{1\over 4} \  \bigg\{\| <v>^\gamma |v\wedge D_v|^2 f \|^{1\over 8} \ \| <v>^{\gamma +2} f\|^{1\over 8} + \| <v>^\gamma |v\wedge D_v|^2 f \|^{1\over 4} + \| <v>^{\gamma +2} f\|^{1\over 4}\bigg\} $$
and all in all
$$\| <v>^{\gamma\over 3} <v\wedge D_x>^{2\over 3} f\| \lesssim \| <v>^{\gamma\over 3} f\| + \| <v>^{\gamma\over 3} |D_x|^{2\over 3} f\| $$
$$+ \| g\|^{1\over 2}  \  \| <v>^{\gamma\over 3} |D_x|^{2\over 3} f\|^{1\over 4} \  \bigg\{ \| <v>^\gamma |v\wedge D_v|^2 f \|^{1\over 8} \  \| <v>^{\gamma +2} f\|^{1\over 8} + \| <v>^\gamma |v\wedge D_v|^2 f \|^{1\over 4} + \| <v>^{\gamma +2} f\|^{1\over 4}\bigg\} $$
that is also
$$\| <v>^{\gamma\over 3} <v\wedge D_x>^{2\over 3} f\| \lesssim \| <v>^{\gamma\over 3} f\| + \| <v>^{\gamma\over 3} |D_x|^{2\over 3} f\| $$
$$+ \| g\|^{1\over 2}  \  \| <v>^{\gamma\over 3} |D_x|^{2\over 3} f\|^{1\over 4} \  \bigg\{  \| <v>^\gamma |v\wedge D_v|^2 f \|^{1\over 4} + \| <v>^{\gamma +2} f\|^{1\over 4}\bigg\} .$$

Using Lemma \ref{lemme-2}, it follows that
$$\| <v>^{\gamma\over 3} <v\wedge D_x>^{2\over 3} f\| \lesssim \| <v>^{\gamma\over 3} f\| + \bigg\{  \| <v>^{\gamma -1} f\|^{1\over 3} + \|<v>^\gamma |D_v|^2 f\|^{1\over 3} \bigg\} \| g\|^{2\over 3}$$
$$+ \| g\|^{2\over 3}  \  \bigg\{  \| <v>^{\gamma -1} f\|^{1\over {12}} + \|<v>^\gamma |D_v|^2 f\|^{1\over {12}} \bigg\} \  \bigg\{ \| <v>^\gamma |v\wedge D_v|^2 f \|^{1\over {4}} + \| <v>^{\gamma +2} f\|^{1\over 4} \bigg\}$$
and thus
$$\| <v>^{\gamma\over 3} <v\wedge D_x>^{2\over 3} f\| \lesssim \| <v>^{\gamma\over 3} f\| + $$
$$ + \| g\|^{2\over 3}\  \bigg\{  \| <v>^{\gamma +2} f\|^{1\over 3} + \|<v>^\gamma |D_v|^2 f\|^{1\over 3} + \|<v>^\gamma |v\wedge D_v|^2 f\|^{1\over 3} \bigg\} ,$$
which concludes the proof.

\end{proof}

\section{Scalar Products between elements of $\cL$ and the transport part}

The main idea is to get an estimate on the square of the norms of each $\cL_i$, and then conclude with the Lemma from the previous sections.

\smallskip

{\bf First Step}

 We shall first of all start first by getting an estimate on $\| \cL_3 f \|$, that is on  $\| <v>^{\gamma +2}f\|$, which is the easiest to obtain. It will be also helpful in order to control other scalar products.

\begin{lemm}\label{lemme-4} We have
$$\| <v>^{\gamma +2} f\| \lesssim \| h\| +\| f\| .$$
\end{lemm}
\begin{proof}
We take the equation, multiply by $<v>^{\gamma +2}f$ (or by $\cL_3 f$) and integrate to get
\begin{equation}\label{ineq-6}
(\lambda (v) \nabla_v f , \nabla_v (<v>^{\gamma +2} f)) + (\mu(v) (v\wedge \nabla_v f) , (v\wedge \nabla_v (<v>^{\gamma +2} f))) + \| <v>^{\gamma +2}f\|^2 \lesssim \| h\| \ \| <v>^{\gamma +2}f \| .
\end{equation}
and we consider the first and second terms on the l.h.s. of this inequality, denoted by $J$ and $K$ respectively, that we need to bound from above (and removing any positive contribution). 

The first term on the l.h.s. of \eqref{ineq-6} is
$$J= (\lambda (v) \nabla_v f , \nabla_v (<v>^{\gamma +2} f)) = - ( <v>^{\gamma +2} \nabla_v . (\lambda (v) \nabla_v f) ,f)$$
$$ = - ( \nabla_v . (<v>^{\gamma +2}  \lambda (v) \nabla_v f) ,f) + (\lambda (v) \nabla_v <v>^{\gamma +2}  \nabla_v f ,f) .$$

We can forget the first term because it is positive, i.e. let
$$J \simeq (\lambda (v) \nabla_v <v>^{\gamma +2} , \nabla_v f) \simeq ( \tilde\lambda v.\nabla_v f ,f)$$
with $\tilde \lambda = \lambda <v>^{\gamma}$. Now
$$( \tilde\lambda v.\nabla_v f ,f) = - (f, \nabla_v . (v\tilde \lambda f)) = - (f, \nabla_v . (v\tilde \lambda ) f) - ( \tilde\lambda v.\nabla_v f ,f) .$$
and thus
$$J \simeq  (f, \nabla_v . (v\tilde \lambda ) f) \lesssim ( <v>^{2\gamma } f ,f) = (<v>^{\gamma +2} f , <v>^{\gamma -2}f ) \lesssim \varepsilon \| <v>^{\gamma +2} f\|^2 +C_\varepsilon \| f\|^2 .$$

Similarly, the second term on the l.h.s. of \eqref{ineq-6}  is
$$K=  (\mu(v) (v\wedge \nabla_v f) , (v\wedge \nabla_v (<v>^{\gamma +2} f)))$$

that we can write as
$$ K = ( \mu (v) V_j f ,V_j (<v>^{\gamma +2} f))  = ( \mu (v) V_j f ,V_j [<v>^{\gamma +2}] f) + ( \mu (v) V_j f ,<v>^{\gamma +2}V_j[ f])$$
and again we can forget the second term since it is positive, to write
$$K\simeq ( \mu (v) V_j f ,V_j [<v>^{\gamma +2}] f) .$$

Since  $\beta_j = \mu V_j [<v>^{\gamma +2}] =0$, it follows that
$$K \simeq (\beta_j V_j f,f) = - (f, V_j [\beta_j f)) = - (f, \beta_j V_j f)) - (f, V_j [\beta_j] f)$$

and therefore $K\simeq (f, V_j [\beta_j] f) =0$. In view of the estimates on $J$ and $K$ just obtained, we can go back to \eqref{ineq-6}, ending the proof.

\end{proof}

As a corollary of the proof, note that we have also (though we do not use it)
\begin{equation}\label{estimations-1}
\| <v>^{\gamma +1} (v\wedge \nabla_v f) \| + \| <v>^{\gamma +1} \nabla_v f \| \lesssim \| h\| +\| f\| .
\end{equation}

\smallskip

{\bf Step 2: A preliminary inequality} 

Below, we set $g= h -\cL (f)$. We start from
$$\cL (f) = -\partial_t f -v.\nabla_x f +h .$$

Then
$$ \| \cL (f)\|^2 = - (v.\nabla_x f , \cL (f) ) + (h, \cL (f))$$
$$ \lesssim (v.\nabla_x f , \cL (f) )  + \| h\|\  \|\cL (f)\|$$
By expanding the square, and using Holder inequality with a parameter $\varepsilon$, we obtain (recall that we have already obtained a control for $\| \cL_3 f\|^2$ from Lemma \ref{lemme-4})
\begin{equation}\label{step-0}
2(\cL_1 (f) , \cL_2 (f) + \| \cL_1 f\|^2 + \| \cL_2 f\|^2 \lesssim (v.\nabla_x f , \cL_1 (f) +\cL_2 (f) ) + \| h\|^2 + \| f\|^2  .
\end{equation}

{\bf Step 3: Scalar product with the transport operator}

Now we compute the scalar product of $v.\nabla_x f$ with $\cL_1(f) +\cL_2 (f)$ which appears in \eqref{step-0}, for which we have

\begin{lemm}\label{lemme-5} With the above notations, we have
\begin{equation}\left\{\begin{array}{c} 
(v.\nabla_x f ,\cL_1 (f) +\cL_2 (f) )\lesssim \| f\|^2 + \| h\|^2 + \\
 \| g\|\  \bigg\{  \| f\|^{1\over 2} + \| <v>^{\gamma +2} f\|^{1\over 2} + \|<v>^\gamma |D_v|^2 f\|^{1\over 2} + \|<v>^\gamma |v\wedge D_v|^2 f\|^{1\over 2}  \bigg\} \end{array}\right.
 \end{equation}
where again recall that $g = h- \cL(f)$.
\end{lemm}
\begin{proof}  We have
$$  (v.\nabla_x f, \cL_1 (f) + \cL_2(f))$$
$$ = - (v.\nabla_x f, \nabla_v .\lambda (v) \nabla_v f(v)) -  (v.\nabla_x f, (v\wedge \nabla_v) . \mu (v) (v\wedge \nabla_v ) (f))$$
$$ = (\partial_{v_j} (v.\nabla_x f) , \lambda (v) \partial_{v_j} f ) + ((v\wedge \nabla_v)(v.\nabla_x f) , \mu (v) (v\wedge \nabla_v f))$$
$$ = (\nabla_x f , \lambda (v) \nabla_v f)) +(V_j (v.\nabla_x f) , \mu(v) V_j (f)) .$$

Now note that $V_j (v_i) = (v\wedge e_j).\nabla_v (v_i )= (v\wedge e_j).e_i$. Thus
$$ (v.\nabla_x f , Q(f)) = (\nabla_x f , \lambda (v) \nabla_v f)) +((v\wedge e_j).\nabla_x f, \mu(v) V_j (f)) $$
$$ = (\nabla_x f , \lambda (v) \nabla_v f))  - (v\wedge\nabla_x f , \mu(v) v\wedge \nabla_v f) =J_1 +J_2 .$$
The term $J_1$ is estimated as
$$J_1 = (<v>^{\gamma\over 3} \nabla_x |D_x|^{-1\over 3} f , \lambda (v)<v>^{-\gamma\over 3} \nabla_v |D_x|^{1\over 3} f) .$$

Noticing that $\tilde \lambda (v) = \lambda (v) .<v>^{-\gamma\over 3} \lesssim <v>^{2\gamma\over 3}$, we get
$$J_1 \lesssim \| <v>^{\gamma\over 3} |D_x|^{2\over 3} f\| \  \| \lambda (v)<v>^{-\gamma\over 3} \nabla_v |D_x|^{1\over 3} f\| ,$$
and we are reduced to study
$$\| \lambda (v)<v>^{-\gamma\over 3} \nabla_v |D_x|^{1\over 3} f\| .$$

We take $|D_x|^{1\over 3}$ of the equation, multiply by $\lambda <v>^{-2\gamma\over 3} =\lambda_1$, multiply by $|D_x|^{1\over 3}f$ and integrate to get
$$- ( \lambda_1 \nabla_v . \lambda \nabla_v |D_x|^{1\over 3}f ,|D_x|^{1\over 3}f  ) - ( \lambda_1 (v\wedge \nabla_v )\mu (v) (v\wedge \nabla_v)  |D_x|^{1\over 3}f ,|D_x|^{1\over 3}f  ) + (\lambda_1 F |D_x|^{2\over 3} f, f)= ( \lambda_1 |D_x|^{1\over 3} h, |D_x|^{1\over 3} f) .$$

We can forget the third term on the left hand side also. The first term on the right hand side is
$$- ( \lambda_1 \nabla_v . \lambda \nabla_v |D_x|^{1\over 3}f ,|D_x|^{1\over 3}f  )  = + (\lambda \nabla_v |D_x|^{1\over 3}f , \nabla_v \lambda_1 |D_x|^{1\over 3}f  )  + ( \lambda_1\lambda \nabla_v |D_x|^{1\over 3}f ,\nabla_v |D_x|^{1\over 3}f  ) .$$

The second term is what we want to estimate. So we need to upper bound the first one. Set $\beta = \lambda\nabla_v \lambda_1$ as a vector field. Then this is of the form (and by symmetry)
$$ (\beta.\nabla_v g, g) \simeq (g, g\ div_v \beta) .$$

Now 
$$div \ \beta \sim \lambda \Delta_v \lambda_1 + \nabla_v \lambda .\nabla_v \lambda_1$$

Since in fact $\lambda \sim <v>^\gamma$ and $\lambda_1 \sim <v>^{\gamma\over 3}$, then $\beta\sim <v>^{4(\gamma /3) -1}$ and thus $div \ \beta \sim <v>^{4(\gamma /3) -2}$,  we have something similar to
$$ (<v>^{4(\gamma /3) -2} |D_x|^{1\over 3}f  , |D_x|^{1\over 3}f )  \lesssim \| <v>^{\gamma\over 3} |D_x|^{2\over 3}f \| \ \| <v>^{\gamma -2} f\| .$$

This computation also adapts to the other term, and we get
$$\| \lambda (v)<v>^{-\gamma\over 3} \nabla_v |D_x|^{1\over 3} f\|^2 \lesssim  \| <v>^{\gamma\over 3} |D_x|^{2\over 3}f \| \ \bigg\{ \| <v>^{\gamma -2} f\| + \| h\| \bigg\}$$
that is
$$\| \lambda (v)<v>^{-\gamma\over 3} \nabla_v |D_x|^{1\over 3} f\| \lesssim  \| <v>^{\gamma\over 3} |D_x|^{2\over 3}f \|^{1\over 2} \ \bigg\{ \| <v>^{\gamma -2} f\|^{1\over 2} + \| h\|^{1\over 2} \bigg\}$$
and thus
$$J_1 \lesssim \| <v>^{\gamma\over 3} |D_x|^{2\over 3}f \|^{3\over 2} \ \bigg\{  \| <v>^{\gamma -2} f\|^{1\over 2} + \| h\|^{1\over 2} \bigg\} .$$

Using Lemma \ref{lemme-2} (with $g= h-\cL (f)$), one has
$$ \| <v>^{\gamma\over 3} |D_x|^{2\over 3} f\|^{3\over 2} \lesssim \bigg\{  \| <v>^{\gamma -1} f\|^{1\over 2} + \|<v>^\gamma |D_v|^2 f\|^{1\over 2} \bigg\} \ \| g\|$$
and we get
$$J_1\lesssim  \bigg\{ \| <v>^{\gamma -1} f\|^{1\over 2} + \|<v>^\gamma |D_v|^2 f\|^{1\over 2} \bigg\} \  \| g\| \ \bigg\{ \| <v>^{\gamma -2} f\|^{1\over 2} + \| h\|^{1\over 2} \bigg\} .$$

We now turn to $J_2$, recalling that $J_2 = - (v\wedge\nabla_x f , \mu(v) v\wedge \nabla_v f)$, and we proceed as for $J_1$. We write
$$J_2 = - (<v>^{\gamma \over 3}v\wedge\nabla_x |v\wedge \nabla_x|^{-1\over 3} f , \mu(v)<v>^{-\gamma\over 3} v\wedge \nabla_v |v\wedge \nabla_x|^{1\over 3}f)$$
$$\lesssim \| <v>^{\gamma\over 3} |v\wedge D_x|^{2\over 3} f\| \ \| \mu <v>^{-\gamma \over 3} v\wedge\nabla_v  |v\wedge \nabla_x|^{1\over 3}f\|$$

and therefore we need to study
$$\| \mu <v>^{-\gamma \over 3} v\wedge\nabla_v  |v\wedge \nabla_x|^{1\over 3}f\| .$$

We take $|v\wedge \nabla_x |^{1\over 3}$ of the equation, multiply by $\mu <v>^{-2\gamma\over 3} = \mu_1$, multiply by $f$ and integrate to get
$$- ( \mu_1 \nabla_v . \lambda \nabla_v |v\wedge D_x|^{1\over 3}f ,|v\wedge D_x|^{1\over 3}f  ) - ( \mu_1 (v\wedge \nabla_v )\mu (v) (v\wedge \nabla_v)  |v\wedge D_x|^{1\over 3}f ,|v\wedge D_x|^{1\over 3}f  ) = ( \mu_1 |v\wedge D_x|^{1\over 3} h, |v\wedge D_x|^{1\over 3} f) .$$

Above  the last term involving $F$ was omitted. We look for the second term:
$$ - ( \mu_1 (v\wedge \nabla_v )\mu (v) (v\wedge \nabla_v)  |v\wedge D_x|^{1\over 3}f ,|v\wedge D_x|^{1\over 3}f  )$$
$$ - ( \mu_1 V_j \mu V_j  |v\wedge D_x|^{1\over 3}f ,|v\wedge D_x|^{1\over 3}f  )$$
$$ = ( \mu_1 \mu V_j  |v\wedge D_x|^{1\over 3}f, V_j |v\wedge D_x|^{1\over 3}f  ) + (V_j [\mu_1] \mu V_j  |v\wedge D_x|^{1\over 3}f,|v\wedge D_x|^{1\over 3}f ) .$$

The first term of this equality is the one we are looking for. For the second one, set $\beta_j = V_j [\mu_1] \mu$, then this is of the form (and by anti symmetry):
$$ (\beta_jV_j g, g) \simeq (V_j [\beta_j] g, g)$$

Now $\mu_1 \sim <v>^{\gamma\over 3}$. Then $V_j [ \mu_1]\sim <v>^{\gamma\over 3}$. Then $\beta_j \sim <v>^{4\gamma\over 3}$ and so is $V_j [\beta_j]$ (in general).

Thus
$$(\beta_jV_j g, g) \simeq (V_j [\beta_j] g, g) \sim (<v>^{\gamma\over 3} |v\wedge D_x|^{2\over 3}f  ,<v>^\gamma f) .$$

Therefore, we find
$$\| \mu <v>^{-\gamma \over 3} v\wedge\nabla_v  |v\wedge \nabla_x|^{1\over 3}f\| \lesssim \| <v>^{\gamma\over 3} |v\wedge D_x|^{2\over 3}f \|^{1\over 2}\  \bigg\{  \| <v>^\gamma f\|^{1\over 2} + \| h\|^{1\over 2}\bigg\}$$
and thus
$$J_2 \lesssim \| <v>^{\gamma\over 3} |v\wedge D_x|^{2\over 3}f \|^{3\over 2} \ \bigg\{  \| <v>^\gamma f\|^{1\over 2} + \| h\|^{1\over 2}\bigg\} .$$

From Lemma \ref{lemme-3}, it follows that
$$\| <v>^{\gamma\over 3} <v\wedge D_x>^{2\over 3} f\|^{3\over 2} \lesssim \| <v>^{\gamma\over 3} f\|^{3\over 2} + $$
$$ + \| g\| \ \bigg\{  \| <v>^{\gamma +2} f\|^{1\over 2} + \|<v>^\gamma |D_v|^2 f\|^{1\over 2} + \|<v>^\gamma |v\wedge D_v|^2 f\|^{1\over 2} \bigg\} .$$

Therefore
$$J_2 \lesssim \bigg\{ \| <v>^\gamma f\|^{1\over 2} + \| h\|^{1\over 2}\bigg\}  \times $$
$$\Bigg\{ \| <v>^{\gamma\over 3} f\|^{3\over 2}  + \| g\| \ \bigg\{  \| <v>^{\gamma +2} f\|^{1\over 2} + \|<v>^\gamma |D_v|^2 f\|^{1\over 2} + \|<v>^\gamma |v\wedge D_v|^2 f\|^{1\over 2} \bigg\} \Bigg\}$$
and thus
$$J_1 +J_2 \lesssim  \bigg\{  \| <v>^{\gamma -1} f\|^{1\over 2} + \|<v>^\gamma |D_v|^2 f\|^{1\over 2} \bigg\} \ \| g\| \  [\bigg\{ \| <v>^{\gamma -2} f\|^{1\over 2} + \| h\|^{1\over 2} ]\bigg\}+$$
$$\bigg\{ \| <v>^\gamma f)\|^{1\over 2} + \| h\|^{1\over 2}\bigg\} \times $$
$$\Bigg\{ \| <v>^{\gamma\over 3} f\|^{3\over 2}  + \| g\| \ \bigg\{  \| <v>^{\gamma +2} f\|^{1\over 2} + \|<v>^\gamma |D_v|^2 f\|^{1\over 2} + \|<v>^\gamma |v\wedge D_v|^2 f\|^{1\over 2} \bigg\} \Bigg\} .$$

Using Lemma \ref{lemme-4}, we get (though of course it is not optimal), and taking into account that $\gamma \leq 1$ and Lemma \ref{lemme-1}, using the fact that ${\gamma\over 3} \leq {\gamma\over 2}+1$
$$J_1+J_2 \lesssim  \bigg\{  \| f\|^{1\over 2} + \|<v>^\gamma |D_v|^2 f\|^{1\over 2} \bigg\} \ \| g\| \  \bigg\{ \| f\|^{1\over 2} + \| h\|^{1\over 2} \bigg\} +$$
$$\bigg\{ \| f\|^{1\over 2} + \| h\|^{1\over 2}\bigg\} \times $$
$$\Bigg\{ \|f\|^{3\over 2} + \|h\|^{3\over 2} + \| g\| \ \bigg\{  \| <v>^{\gamma +2} f\|^{1\over 2} + \|<v>^\gamma |D_v|^2 f\|^{1\over 2} + \|<v>^\gamma |v\wedge D_v|^2 f\|^{1\over 2} \bigg\} \Bigg\} $$

which simplifies to yield
\begin{equation}\label{estimation-cruciale}\left\{\begin{array}{c} 
(v.\nabla_x f ,\cL_1 (f) +\cL_2 (f) )\lesssim \| f\|^2 + \| h\|^2 + \\
 \| g\| \bigg\{  \| f\|^{1\over 2} + \| <v>^{\gamma +2} f\|^{1\over 2} + \|<v>^\gamma |D_v|^2 f\|^{1\over 2} + \|<v>^\gamma |v\wedge D_v|^2 f\|^{1\over 2}  \bigg\} \end{array}\right.
 \end{equation}
where again recall that $g = h- \cL(f)$.
\end{proof}

Note the exponents on the r.h.s. of this inequality which are less than $2$. 

\smallskip

{\bf Step 4: The remaining scalar product}

We consider the last scalar product which appears in \eqref{step-0} (recall that it is on the l.h.s of the inequality we want to control and therefore we can forget any positive term)
$$(\cL_1 (f) ,\cL_ 2(f) = ( \nabla_v . \lambda (v)\nabla_v f , (v\wedge \nabla_v) \mu (v) (v\wedge \nabla_v )f) .$$

If we introduce $B_k = \sqrt\lambda \partial_{v_k}$, and $W_j = \sqrt\mu V_j$, then this term is also

$$ (B^\ast_k B_k f , W^\ast_j W_j f) = (B_k f , B_k W^\ast_j W_j f) = (B_k f, [B_k, W^\ast_j] W_j f) + (B_k f, W^\ast_j B_k W_j f)$$
$$ = (B_k f, [B_k, W^\ast_j] W_j f) + (W_jB_k f, B_k W_j f)$$
$$ = (B_k f, [B_k, W^\ast_j] W_j f) + (W_jB_k f, [B_k, W_j] f) + (W_j B_k f, W_jB_k f) .$$

If we exchange the role of $B_k$ and $W_j$ we find
$$ = (W_j f, [W_j, B^\ast_k] B_k f) + (B_kW_j f, [W_j, B_k] f) + (B_k W_j f, B_kW_j f)$$
$$ = ([W_j, B^\ast_k] ^\ast W_j f, B_k f) + (B_kW_j f, [W_j, B_k] f) + (B_k W_j f, B_kW_j f)$$
$$ = ([B_k, W^\ast_j] W_j f, B_k f) + (B_kW_j f, [W_j, B_k] f) + (B_k W_j f, B_kW_j f) .$$

Thus adding two lines we find
$$2 (B^\ast_k B_k f , W^\ast_j W_j f) = 2 ([B_k, W^\ast_j] W_j f, B_k f)  +(W_jB_k f, [B_k, W_j] f) - (B_kW_j f, [B_k, W_j] f) $$
$$+ (W_j B_k f, W_jB_k f) +(B_k W_j f, B_kW_j f)$$

so we have
\begin{equation}\label{produit-croise}
\left\{ \begin{array}{c}  2 (\cL_1 (f), \cL_2 (f) ) 
= 2 (B^\ast_k B_k f , W^\ast_j W_j f) \\ = 2 ([B_k, W^\ast_j] W_j f, B_k f) - ([B_k, W_j]f,  [B_k, W_j] f) \
+ (W_j B_k f, W_jB_k f) +(B_k W_j f, B_kW_j f) . \end{array}\right.
\end{equation}

Note that the last two terms are positive (so we can forget them). We compute 
$$[B_k ,W_j ] (f) = [ \sqrt \lambda \partial_{v_k} , \sqrt\mu V_j ] (f) = \sqrt \lambda \partial_{v_k}  [ \sqrt\mu V_j  (f) ] - \sqrt\mu V_j [  \sqrt \lambda \partial_{v_k}  (f)]$$
$$ = -\sqrt \lambda \partial_{v_k}  [ \sqrt\mu (\nabla_vf \wedge e_j).v] - \sqrt \mu (v\wedge e_j).\nabla_v [   \sqrt \lambda \partial_{v_k}  (f)]$$
$$  \sqrt\lambda \partial_{v_k} [\sqrt \mu] V_j (f) +\sqrt\lambda \sqrt \mu V_j (\partial_{v_k} f)  +\sqrt\lambda \sqrt\mu (e_k \wedge e_j) .\nabla_v f $$
$$ - \sqrt \mu (v\wedge e_j) . \nabla_v [\sqrt\lambda ]\partial_{v_k}  (f) - \sqrt\lambda \sqrt \mu (v\wedge e_j) .\nabla_v \partial_{v_k} f$$
$$ =  \sqrt\lambda \partial_{v_k} [\sqrt \mu] V_j (f)  +\sqrt\lambda \sqrt\mu (e_k \wedge e_j) .\nabla_v f - \sqrt \mu (v\wedge e_j) . \nabla_v [\sqrt\lambda ]\partial_{v_k}  (f)$$
$$ = \sqrt\lambda \partial_{v_k} [\sqrt \mu] V_j (f)  +\sqrt\lambda \sqrt\mu (e_k \wedge e_j) .\nabla_v f - \sqrt \mu V_j [\sqrt\lambda ]\partial_{v_k}  (f)$$

Therefore
$$\| [B_k ,W_j ] (f)\|^2 \lesssim \| \sqrt\lambda \partial_{v_k} [\sqrt \mu] V_j (f)\|^2  + \| \sqrt\lambda \sqrt\mu (e_k \wedge e_j) .\nabla_v f\|^2 + \| \sqrt \mu V_j [\sqrt\lambda ]\partial_{v_k}  (f)\|^2 .$$

We note that the weight for the first term is similar to $<v>^{\gamma -1}$, for the second one to $<v>^\gamma$ and for the last to $<v>^\gamma$ also. Thus
$$\| [B_k ,W_j ] (f)\|^2 \lesssim  \| <v>^{\gamma} \nabla_v f\|^2$$
$$ = (<v>^{2\gamma} \nabla_v f , \nabla_v f ) = - (<v>^{2\gamma} \Delta_v f , f) - (\nabla_v <v>^{2\gamma} .\nabla_v f ,f) $$
and by symmetry
$$ \lesssim (<v>^{2\gamma} \Delta_v f , f) +  (\Delta_v <v>^{2\gamma} f, f).$$
Thus
$$ \| [B_k ,W_j ] (f)\|^2 \lesssim \varepsilon \| <v>^\gamma |D_v|^2 f\|^2 +C_\varepsilon \| <v>^{{\gamma\over 2} +1 } f\|^2 .$$

For the term in \eqref{produit-croise} involving $2 ([B_k, W^\ast_j] W_j f, B_k f)$, we can substract $W^\ast_j$ to write
$$2 ([B_k, W^\ast_j] W_j f, B_k f) = 2 ([B_k, W^\ast_j -W_j] W_j f, B_k f) + 2 ([B_k, W_j] W_j f, B_k f)$$
$$ = 2 ([B_k, V_j (\sqrt\mu)] W_j f, B_k f) + 2 ([B_k, W_j] W_j f, B_k f)$$

Now $[B_k, V_j (\sqrt\mu )] (f) = \sqrt\lambda\partial_{v_k} [ V_j (\sqrt\mu )] f$ and we note that the weight is $<v>^{\gamma -1}$. Thus the first term is estimated by something similar to
$$(<v>^{\gamma -1} W_j f, B_k f ) \lesssim \|<v>^{{\gamma \over 2} -1/2} W_j f\|^2 + \| <v>^{{\gamma \over 2} -1/2}B_k f\|^2 \lesssim \| f\| \  \|h\|$$

For the other remaining term $2 ([B_k, W_j] W_j f, B_k f)$, we have in view of the commutators that it looks as something similar to
$$(<v>^{\gamma -1} V_j \sqrt\mu V_j f, \sqrt \lambda \partial_{v_k} f) + (<v>^\gamma (e_k\wedge e_j).\nabla_v \sqrt \mu V_j f, \sqrt \lambda \partial_{v_k} f)  + (<v>^\gamma \partial_{v_k} \sqrt\mu V_j f, \sqrt \lambda \partial_{v_k} f) $$
and thus it looks as
$$<v>^{2\gamma -1} V_j^2 f, \partial_{v_k} f) + (<v>^{2\gamma} V_j , \Delta_v f ) $$
$$\lesssim \| <v>^\gamma V^2_j f\| . \| <v>^{\gamma -1} \nabla_v f\| + \|<v>^\gamma |D_v|^2 f\| \| <v>^\gamma V_j f\|$$
$$\lesssim \| <v>^\gamma V^2_j f\| . \| <v>^{\gamma \over 2} \nabla_v f\| + \|<v>^\gamma |D_v|^2 f\| \| <v>^\gamma V_j f\|$$
$$\lesssim \varepsilon \| <v>^\gamma V^2_j f\|^2  +C_\varepsilon  \| <v>^{\gamma \over 2} \nabla_v f\|^2 + \varepsilon\|<v>^\gamma |D_v|^2 f\|^2 +C_\varepsilon \| <v>^\gamma V_j f\|^2 .$$

But we can also write that
$$\| <v>^\gamma V_j f\|^2 = (<v>^{2\gamma} V_jf, V_j f ) = (<v>^\gamma V^2_j f , <v>^\gamma f) \lesssim \| <v>^\gamma |v\wedge D_v|^2 f \| \ \| <v>^\gamma f\|$$
thus
$$  \|<v>^\gamma |D_v|^2 f\| \ \| <v>^\gamma V_j f\| \lesssim \|<v>^\gamma |D_v|^2 f\|^{3\over 2} \ \| <v>^\gamma f\|^{1\over 2}$$

and by using Holder with exponent and a small exponent $\varepsilon$ again, we find that it is less than
$$\varepsilon \|<v>^\gamma |D_v|^2 f\|^{2}  +C_\varepsilon \|<v>^\gamma f\|^2 .$$

All in all, we have shown that all terms from the scalar product (but for the positive) in \eqref{produit-croise} can be absorbed with previous ones. That is the scalar product on the l.h.s. of \eqref{step-0} can be absorbed with all the other terms, taking into account also \eqref{estimation-cruciale}. This is enough to conclude the proof of the main Theorem, in view of the exponents appearing on the r.h.s.

\begin{rema}

We make some comments about the case of the Cauchy problem, when working say on a time interval $(0,T)$ (possibly with $T= +\infty$), with an initial value at time $t=0$ given by $f_0$. 

Let $\phi_1 =\phi_1 (t)$ be a smooth and compactly supported positive function being one for small positive values of time. Let $\phi_2 = \phi_2 (t)$ be again a smooth and positive function, but being one for larger values of $t$ and such that $\phi_1 +\phi_2 \sim 1$. Let $f_1 = \phi_1 f$ and $f_2 = \phi_2 f$. We shall also assume that eventually $\phi_2$ has a compact support on the right of the real axis.

We first consider $f_2$. Then
\begin{equation}\label{cauchy-eq1}
\partial_t f_2 +v.\nabla_x f_2 - \nabla_v .\lambda (v) \nabla_v f_2 - (v\wedge \nabla_v ).\mu (v) (v\wedge \nabla_v f_2) +F(v) f_2(v) = \phi_2 h  + \phi '_2 f  \equiv h_2.
\end{equation}
Now considering the fact that this equation for $f_2$ holds true all over $\RR$, the previous results apply. Of course, we need $f$ to be in $L^2$ in all variables. This statement might be assumed, but note that it also follows from the equation if we assume that $\gamma \geq -2$ by using the coercivity from the third term of the diffusive part. 

We can now turn to $f_1$ which also satisfies:
\begin{equation}\label{cauchy-eq2}
\partial_t f_1 +v.\nabla_x f_1 - \nabla_v .\lambda (v) \nabla_v f_1 - (v\wedge \nabla_v ).\mu (v) (v\wedge \nabla_v f_2) +F(v) f_1(v) = \phi_1 h  + \phi '_1 f \equiv h_1,  \ f_1|_{t=0} = f_0 .
\end{equation}
A careful look at the previous proof shows that if we assume that $\cL_i f_0$ belongs to $L^2$, then we have again the same conclusion as in the main theorem. 

Finally, a last (formal) comment can be made if we assume $h=0$ from the beginning. In that case, as it is standard in semi group theory, we can get formally more regularity. As explained above, we assume $\gamma \geq -2$ to ensure that we have automatically an $L^2$ bound. Then following formally the arguments from \cite{caze},  we see that, using only norms w.r.t. $(x,v)$ variables
$$\| ( -v.\nabla_x f -\cL (f) ) f(t) \| \leq {1\over{t\sqrt 2}} \| f_0\|$$
By using the main Theorem (without time dependence), we see that we have hypoelliptic results on $f(t)$ in terms of the initial data, and with a singular behavior as times tends to $0$.

\end{rema}

\end{document}